\documentclass[12pt]{amsart}
\usepackage{amsfonts, amssymb, latexsym, mathrsfs}
\usepackage{mathtools}
\usepackage[usenames, dvipsnames]{color}
\usepackage{ytableau}
\usepackage{xcolor}
\usepackage{tikz}

\newtheorem{theorem}{Theorem}[section]
\newtheorem{proposition}[theorem]{Proposition}

\newtheorem{claim}[theorem]{Claim}
\newtheorem*{claim*}{Claim}
\newtheorem{corollary}[theorem]{Corollary}
\newtheorem{Main Conjecture}[theorem]{Main Conjecture}
\newtheorem{conjecture}[theorem]{Conjecture}
\newtheorem*{principle*}{Atiyah-Bott Combinatorial Dream (A$\cdot$B$\cdot$C$\cdot$D)}

\theoremstyle{remark}

\newtheorem{example}[theorem]{Example}

\theoremstyle{plain}

\newtheorem*{Example*}{Example}

\newcommand*{\rom}[1]{\expandafter\@slowromancap\romannumeral #1@}

\newcommand{\cellsizeL}{19}
\newcommand{\cellsizeS}{14}
\newlength{\cellszL} 
\newsavebox{\cellL}
\sbox{\cellL}{\begin{picture}(\cellsizeL,\cellsizeL)
\put(0,0){\line(1,0){\cellsizeL}}
\put(0,0){\line(0,1){\cellsizeL}}
\put(\cellsizeL,0){\line(0,1){\cellsizeL}}
\put(0,\cellsizeL){\line(1,0){\cellsizeL}}
\end{picture}}
\newcommand\cellifyL[1]{\def\thearg{#1}\def\nothing{}%
\ifx\thearg\nothing
\vrule width0pt height\cellszL depth0pt\else
\hbox to 0pt{\usebox{\cellL} \hss}\fi%
\vbox to \cellszL{
\vss
\hbox to \cellszL{\hss$#1$\hss}
\vss}}
\newcommand\tableauL[1]{\vtop{\let\\\cr
\baselineskip -16000pt \lineskiplimit 16000pt \lineskip 0pt
\ialign{&\cellifyL{##}\cr#1\crcr}}}

\newlength{\cellszS} 
\newsavebox{\cellS}
\sbox{\cellS}{\begin{picture}(\cellsizeS,\cellsizeS)
\put(0,0){\line(1,0){\cellsizeS}}
\put(0,0){\line(0,1){\cellsizeS}}
\put(\cellsizeS,0){\line(0,1){\cellsizeS}}
\put(0,\cellsizeS){\line(1,0){\cellsizeS}}
\end{picture}}
\newcommand\cellifyS[1]{\def\thearg{#1}\def\nothing{}%
\ifx\thearg\nothing
\vrule width0pt height\cellszS depth0pt\else
\hbox to 0pt{\usebox{\cellS} \hss}\fi%
\vbox to \cellszS{
\vss
\hbox to \cellszS{\hss$#1$\hss}
\vss}}
\newcommand\tableauS[1]{\vtop{\let\\\cr
\baselineskip -16000pt \lineskiplimit 16000pt \lineskip 0pt
\ialign{&\cellifyS{##}\cr#1\crcr}}}
\newcommand{\gap}{\hspace{1in} \\ \vspace{-.2in}}

\newcommand{\excise}[1]{}%{$\star$\textsc{#1}$\star$}

\title{The A$\cdot$B$\cdot$C$\cdot$D\MakeLowercase{s} of Schubert calculus}
\author{Colleen Robichaux}
\author{Harshit Yadav}
\author{Alexander Yong}
\address{Dept. of Mathematics, University of Illinois at Urbana-Champaign, Urbana, IL 61801}
\email{cer2@illinois.edu, yadav7@illinois.edu, ayong@illinois.edu}

\begin{document}
\pagestyle{plain}

\mbox{}

\begin{abstract}
We collect \emph{Atiyah-Bott Combinatorial Dreams} (A$\cdot$B$\cdot$C$\cdot$Ds) in Schubert calculus. One result
relates equivariant structure coefficients for two isotropic flag manifolds, with consequences to the thesis
 of C.~Monical.  We contextualize using work of
N.~Bergeron-F.~Sottile, S.~Billey-M.~Haiman, P.~Pragacz, and T.~Ikeda-L.~Mihalcea-I.~Naruse. The relation complements a 
theorem of A.~Kresch-H.~Tamvakis in quantum cohomology. Results of A.~Buch-V.~Ravikumar rule out a similar correspondence in $K$-theory.
\end{abstract}
\date{June 9, 2019}
\maketitle

\section{Introduction}\label{sec:1}

\subsection{Conceptual framework}
Each generalized flag variety ${\sf G}/{\sf B}$ has  finitely many orbits under the left action of the (opposite) Borel subgroup ${\sf B_{-}}$ of a complex reductive Lie group ${\sf G}$. They are indexed by elements $w$ of the Weyl group ${\mathcal W}\cong N({\sf T})/{\sf T}$, 
where ${\sf T}={\sf B}\cap {\sf B}_{-}$ is a maximal torus.  The \emph{Schubert varieties} are closures $X_w$ of these orbits.
The Poincar\'e duals of the Schubert varieties $\{\sigma_w\}_{w\in {\mathcal W}}$ form a ${\mathbb Z}$-linear basis of the cohomology ring
$H^{\star}({\sf G}/{\sf B})$. The \emph{Schubert structure coefficients} are nonnegative integers, defined by
\[\sigma_u \smallsmile \sigma_v=\sum_{w\in {\mathcal W}} c_{u,v}^{w} \sigma_{w}.\]
Geometrically, $c_{u,v}^w\in {\mathbb Z}_{\geq 0}$ counts intersection points of generic translates of three Schubert varieties. 
The main problem of modern Schubert calculus is to combinatorially explain this positivity. For Grassmannians, this is achieved by the
\emph{Littlewood-Richardson rule} \cite{Fulton}.

The title alludes to a principle, traceable to M.~Atiyah-R.~Bott \cite{Atiyah.Bott},
 that \emph{equivariant cohomology} is a lever on ordinary cohomology. In our case, each $X_w$ is ${\sf T}$-stable, so it admits a class $\xi_w$ in
$H^{\star}_{\sf T}({\sf G}/{\sf B})$, the ${\sf T}$-equivariant cohomology ring of ${\sf G}/{\sf B}$. These classes are a basis for 
$H^{\star}_{\sf T}({\sf G}/{\sf B})$ as a module over the base ring $H^{\star}_{\sf T}(pt)$. If $\Delta=\{\alpha_1,\ldots,\alpha_r\}$ are the simple roots
of the root system $\Phi=\Phi^+\cup \Phi^-$ associated to our pinning of ${\sf G}$,  
$H^{\star}_{\sf T}(pt)\cong {\mathbb Z}[\alpha_1,\ldots,\alpha_r]$.
 
Define the \emph{equivariant Schubert structure coefficient} $C_{u,v}^w\in H^{\star}_{\sf T}(pt)$ by
\begin{equation}
\label{eqn:prod}
\xi_u  \cdot \xi_v =\sum_{w\in {\mathcal W}} C_{u,v}^w\ \xi_w.
\end{equation}
If $\ell(u)+\ell(v)=\ell(w)$, $C_{u,v}^w=c_{u,v}^w$. Thus we have a harder version of the main problem, with $\#\Delta$-many parameters.
Does the equivariant complication make the problem simpler? This study initiates our systematic exploration of the question.

The inclusion $({\sf G}/{\sf B})^{\sf T}\hookrightarrow {\sf G}/{\sf B}$ induces an \emph{injective} map 
\begin{equation}
\label{eqn:injective}
H_{\sf T}^{\star}({\sf G}/{\sf B})\hookrightarrow H_{\sf T}^{\star}(({\sf G}/{\sf B})^{\sf T})\cong\bigoplus_{w\in {\mathcal W}}{\mathbb Z}[\alpha_1,\ldots,\alpha_r].
\end{equation}
Thus, each $\xi_w$ is identified with
a $\#{\mathcal W}$-size list of polynomials $(\xi_w|_v)_{v \in {\mathcal W}}$. Multiplication in $H_{\sf T}^{\star}({\sf G}/{\sf B})$ is thereby pointwise multiplication of these lists. Moreover, there is a formula for the \emph{equivariant restriction} $\xi_w|_v$ due to H.~Andersen-J.~Jantzen-W.~Soergel \cite{Soergel}, and rediscovered by S.~Billey~\cite{Billey}.\footnote{Equivariant restriction is part of \emph{GKM-theory} \cite{GKM}, a subject of extensive investigation; see, e.g., J.~Tymoczko's exposition \cite{Tymoczko}, and the references therein, for an account germane to our discussion. However the case of Schubert varieties is found in work of B.~Kostant-S.~Kumar \cite{KK1,KK2}.} Let $I=s_{\underline{\alpha}_{1}}s_{\underline{\alpha}_{2}}\cdots s_{\underline{\alpha}_{\ell(v)}}$
be a reduced word for $v\in {\mathcal W}$, where $s_{\underline{\alpha}_{i}}$ is the reflection through the hyperplane perpendicular to $\underline{\alpha}_i:=\alpha_{j}\in \Delta$ (for some $j$ depending on $i$). 
Then, 
\begin{equation}
\label{eqn:Billey}
\xi_w|_v=\sum_{J\subseteq I} \prod_I({\underline{\alpha}}_i^{\langle i\in J\rangle} s_{\underline{\alpha}_{i}})\cdot 1;
\end{equation}
\emph{cf.}~\cite[Theorem~1]{Knutson}.
The sum is over subwords $J$ that are reduced words for $w$. 
Also, ${\underline{\alpha}}_i^{\langle i\in J\rangle}$ means ${\underline{\alpha}}_i$ appears only if $i\in J$. Combining (\ref{eqn:Billey}) and (\ref{eqn:injective}) provides the combinatorial definition of $H_{\sf T}^{\star}({\sf G}/{\sf B})$ we use.

This description of  $H_{\sf T}^{\star}({\sf G}/{\sf B})$  permits a \emph{non-positive} linear algebraic computation of $C_{u,v}^w$, see, e.g., \cite[Section~6]{Billey}.  From this perspective, the solved combinatorics
of equivariant restriction and the open problem of Schubert calculus seem far apart. However, we argue using an idealization that the concepts are closer than first supposed:

\begin{principle*}
A combinatorial (positivity) statement true of equivariant restrictions
also holds for Schubert structure coefficients.
\end{principle*}

We begin with retrospective examples:

\begin{itemize} 
\item[(I)] Sometimes combining (\ref{eqn:Billey}) with \emph{basic} Coxeter theory realizes an A$\cdot$B$\cdot$C$\cdot$D. 
\emph{Bruhat order} $\leq$ on ${\mathcal W}$ is geometrically defined by $w\leq v$ if $X_w\supseteq X_v$. Fix a reduced word $I$ of $v$.
The \emph{subword property} of Bruhat order states that $w\leq v$ if and only if there exists a subword $J$ of $I$ that is a reduced word of $w$. Hence
 from (\ref{eqn:Billey}),
 \begin{equation}
 \label{eqn:uppertriangularity}
 \xi_w|_v=0 \text{ \ \ unless $w\leq v$.}
 \end{equation}
  Combining (\ref{eqn:uppertriangularity}) and (\ref{eqn:injective}) gives
 \begin{equation}
 \label{eqn:nonzeroness}
 C_{u,v}^w=0 \text{ unless $u\leq w$ and $v\leq w$.}
 \end{equation}
\item[(II)] The \emph{converse} of A$\cdot$B$\cdot$C$\cdot$D is true. Following \cite[Lemma~1]{Knutson}, by (\ref{eqn:uppertriangularity}) and (\ref{eqn:nonzeroness}),
\[\xi_v|_v\cdot \xi_w|_v=C_{v,w}^v\ \xi_v|_v.\] 
By (\ref{eqn:Billey}), $\xi_v|_v\neq 0$. Hence 
\begin{equation}
\label{eqn:Knutson}
C_{v,w}^{v}=\xi_w|_v.
\end{equation}
This is a tantalizing clue about an eventual combinatorial rule for  $C_{u,w}^{v}$. More concretely, in \cite{Knutson.Tao}, (\ref{eqn:Knutson}) implies a recurrence that, with additional combinatorics, proves an equivariant Littlewood-Richardson rule \emph{sans} symmetric functions.
 \item[(III)] Here is a deep instance (\cite{Billey}, \emph{cf.}~\cite[Section~2]{Tymoczko}). It is textbook \cite[Section~1.7]{Humphreys:reflection} that  
\begin{equation}
\label{eqn:inversionset}
{\sf Inv}(v^{-1}):=\{\alpha\in \Phi^+:v^{-1}(\alpha)\in \Phi^-\}=\{s_{\underline{\alpha}_{1}}s_{\underline{\alpha}_{2}}\cdots s_{\underline{\alpha}_{k-1}}\underline{\alpha}_k: 1\leq k\leq \ell(v)\}. 
\end{equation}
 Since each positive root is a positive linear combination of simples, by (\ref{eqn:inversionset}) and (\ref{eqn:Billey}),
\[
\xi_w|_v \in {\mathbb Z}_{\geq 0}[\alpha_1,\ldots,\alpha_r].\]
Indeed, D.~Peterson conjectured, and W.~Graham \cite{Graham} geometrically proved that 
\begin{equation}
\label{eqn:Graham}
C_{u,v}^{w}\in {\mathbb Z}_{\geq 0}[\alpha_1,\ldots,\alpha_r].
\end{equation} 
\item[(IV)] This is closely related to (III), but is folklore. Since $\ell(v)=\ell(v^{-1})=\#{\sf Inv}(v^{-1})$,
by (\ref{eqn:inversionset}), $s_{\underline{\alpha}_{1}}s_{\underline{\alpha}_{2}}\cdots s_{\underline{\alpha}_{k-1}}\underline{\alpha}_k\in \Phi^+$ are all distinct. Hence from (\ref{eqn:Billey}), $\xi_w|_v$ is \emph{square-free} when expressed in the positive roots. As A.~Knutson (private communication) points out, the proof in \cite{Graham} shows this to be true of $C_{u,v}^w$
as well.
\item[(V)] For any ${\sf G}/{\sf B}$, there is a recurrence, due to B.~Kostant and S.~Kumar to compute $\xi_w|_{v}$; it has an analogue for $C_{u,v}^w$ due to A.~Knutson. See
\cite[Theorem~1]{Knutson:patches} and \cite[Section~1]{Knutson}. In turn, special cases of Knutson's recurrence give ``descent cycling'' relations on the \emph{ordinary} Schubert structure constants \cite{Knutson:descent}. 
\end{itemize}

\subsection{Does  A$\cdot$B$\cdot$C$\cdot$D suggest anything new?} Our main instance is of different flavor than (I)--(V). We relate
all structure coefficients of one isotropic flag variety to those of another; this has consequences. The results are
neither explicit in the literature nor seem well-known. The correspondence generalizes, with a new proof, non-equivariant results of P.~Pragacz \cite{Pragacz} and of N.~Bergeron-F.~Sottile \cite{Bergeron.Sottile} (who rely on S.~Billey-M.~Haiman's work \cite{Billey.Haiman}, which in turn 
generalizes \cite{Pragacz}). We emphasize that the correspondence can also be derived from T.~Ikeda-L.~Mihalcea-H.~Naruse's \cite{IMN}; 
see the discussion of Section~\ref{sec:3}.

Consider the classical groups ${\sf G}={\sf SO}_{2n+1}$ and ${\sf G}={\sf Sp}_{2n}$ of non-simply laced type. These
are automorphism groups preserving a non-degenerate bilinear form $\langle \cdot, \cdot \rangle$. In the former case it is a symmetric form on 
$W={\mathbb C}^{2n+1}$ whereas in the latter case it is skew symmetric form on $W={\mathbb C}^{2n}$. A subspace $V\subseteq W$ is \emph{isotropic}
if, for all $v_1,v_2\in V$, $\langle v_1,v_2\rangle =0$. The maximum dimension of an isotropic space is $n$. Any flag
of isotropic subspaces $\langle 0\rangle \subset F_1\subset F_2\subset \cdots\subset F_n$ extends to a complete flag in $W$ by
$\langle 0\rangle \subset F_1\subset F_2\subset \cdots\subset F_n\subseteq F_n^{\perp}\subset F_{n-1}^{\perp}\subset \cdots \subset F_{1}^{\perp}\subset W$,
where $F_k^{\perp}$ is the orthogonal complement of $F_k$. Then the flag manifolds $X={\sf SO}_{2n+1}/{\sf B}$ and $Y={\sf Sp}_{2n}/{\sf B}$ consist of
complete flags of this form.

The root systems for ${\sf SO}_{2n+1}$ (type $B_n$) and ${\sf Sp}_{2n}$ (type $C_n$) are rank $r=n$. Let $\{\beta_1,\ldots,\beta_n\}$
and $\{\gamma_1,\ldots,\gamma_n\}$ be the simples labelled by their respective Dynkin diagrams
\[\!\!\!\!\!\!\!\!\!\!\!\!\!\!\!\!\!\!\!\!\!\!\!\!\!\!\!\!\begin{picture}(30,10)
\thicklines
\put(0,-6){$1$}
\put(0,3){$\circ$}
\put(5,5){\line(1,0){15}}
\put(5,7){\line(1,0){15}}
\put(8,3){$<$}
\put(19,-6){$2$}
\put(19,3){$\circ$}
\put(24,6){\line(1,0){15}}
\put(38,-6){$3$}
\put(38,3){$\circ$}
\put(44,3){$\cdots$}
\put(60,3){$\cdots$}
\put(65,-6){$n\!-\!1$}
\put(74,3){$\circ$}
\put(79,6){\line(1,0){15}}
\put(93,3){$\circ$}
\put(93,-6){$n$}
\end{picture}
\text{ \ \ \ \ \ \ \  \ \ \ \ \ \ \ \ \ \ \ \ \ \ \ \ \ \ \ and \ \  }
\begin{picture}(30,10)
\thicklines
\put(0,-6){$1$}
\put(0,3){$\circ$}
\put(5,5){\line(1,0){15}}
\put(5,7){\line(1,0){15}}
\put(8,3){$>$}
\put(19,-6){$2$}
\put(19,3){$\circ$}
\put(24,6){\line(1,0){15}}
\put(38,-6){$3$}
\put(38,3){$\circ$}
\put(44,3){$\cdots$}
\put(60,3){$\cdots$}
\put(65,-6){$n\!-\!1$}
\put(74,3){$\circ$}
\put(79,6){\line(1,0){15}}
\put(93,3){$\circ$}
\put(93,-6){$n$}
\end{picture}\] 
The two root systems share the \emph{hyperoctahedral group} ${\mathcal B}_n$ as their common Weyl group. We represent 
${\mathcal B}_n$ as \emph{signed permutations} of $\{1,2,\ldots,n\}$, e.g., $\underline{2} \ 1\  \underline{3}$. Define
\[s(w):=\#\{1\leq i\leq n: w(i)<0\}.\]
Let $\overline{f}\in {\mathbb Z}[\beta_1,\beta_2,\ldots,\beta_n]$ be $f\in {\mathbb Z}[\gamma_1,\gamma_2, \ldots, \gamma_n]$ with $\gamma_1\mapsto 2\beta_1$ and $\gamma_i\mapsto \beta_i$ for $1 < i\leq n$.

\begin{theorem}
\label{thm:main}
$C_{u,v}^w(X)=2^{s(w)-s(u)-s(v)}\overline{C_{u,v}^w(Y)}$.
\end{theorem}
\begin{proof}
This equivalence is from the definitions:
\[\xi_w(Y)|_{x} = \sum_{J\subseteq I} \prod_{I} ({\underline{\gamma}_i}^{\langle i\in J\rangle} s_{\underline{\alpha}_{i}})\cdot 1\iff \overline{\xi_w(Y)|_{x}} = \sum_{J\subseteq I} 2^{\# \{1\in J \}}\prod_{I} ({\underline{\beta}_i}^{\langle i\in J\rangle} s_{\underline{\alpha}_{i}})\cdot 1.\]
The Coxeter combinatorics needed is merely this: since $J$ is a reduced word for $w$, it is true that
$\#\{1\in J\} =s(w)$.
Therefore, 
\begin{equation}
\label{eqn:therefore}
 \overline{\xi_w(Y)|_{x} } = 2^{s(w)} \sum_{J\subseteq I} \prod_{I} ({\underline{\beta}_i}^{\langle i\in J\rangle} s_{\underline{\alpha}_{i}})\cdot 1 = 2^{s(w)} \xi_w(X)|_x;
 \end{equation}
\emph{i.e.}, a ``power of two relationship'' between the restrictions.
 Applying (\ref{eqn:prod}), (\ref{eqn:injective}) and (\ref{eqn:Billey}) to $Y$, 
\[\xi_u(Y)|_x\cdot \xi_v(Y)|_x=\sum_{w\in {\mathcal B}_n} C_{u,v}^w(Y)\ \xi_w(Y)|_x \ \  \ \ \ \forall x\in {\mathcal B}_n\]
\begin{align}\nonumber
\iff & \overline{\xi_u(Y)|_x}\cdot \overline{\xi_v(Y)|_x}=\sum_{w\in {\mathcal B}_n} \overline{C_{u,v}^w(Y)}\ \overline{\xi_w(Y)|_x} \ \  \ \ \ \forall x\in {\mathcal B}_n\\ \nonumber
\iff & (2^{-s(u)}\overline{\xi_u(Y)|_x})\cdot (2^{-s(v)}\overline{\xi_v(Y)|_x})=\sum_{w\in {\mathcal B}_n} 2^{s(w)-s(u)-s(v)}\overline{C_{u,v}^w(Y)}(2^{-s(w)}\overline{\xi_w(Y)|_x}) \ \ \ \forall x\in {\mathcal B}_n\\ \nonumber
\iff & \xi_u(X)|_x\cdot \xi_v(X)|_x=\sum_{w\in {\mathcal B}_n} 2^{s(w)-s(u)-s(v)}\overline{C_{u,v}^w(Y)}\ \xi_w(X)|_x \text{\ \ \ 
$\forall x\in {\mathcal B}_n$
 \ \ \ [by (\ref{eqn:therefore})]}.
\end{align}
We are now done by (\ref{eqn:prod}), (\ref{eqn:injective}) and (\ref{eqn:Billey}) applied to $X$, \emph{i.e.}, uniqueness of the equivariant structure coefficients.
\end{proof}

\begin{example} \label{exa:abc123} Consider $u=3\ \underline{2}\ 1, v=\underline{3}\  \underline{2} \ 1$ and $w=\underline{2}\ \underline{3} \ 1$ in ${\mathcal B}_3$. Then 
$s(w)-s(u)-s(v)=-1$  and
$C_{u,v}^w(Y)=2\gamma_1\gamma_2^2+2\gamma_1\gamma_2\gamma_3+4\gamma_2^3+6\gamma_2^2\gamma_3+2\gamma_2\gamma_3^2$, so
\[\overline{C_{u,v}^w(Y)}=4\beta_1\beta_2^2+4\beta_1\beta_2\beta_3+4\beta_2^3+6\beta_2^2\beta_3+2\beta_2\beta_3^2.\]
We also have
\[{C_{u,v}^w(X)}=2\beta_1\beta_2^2+2\beta_1\beta_2\beta_3+2\beta_2^3+3\beta_2^2\beta_3+\beta_2\beta_3^2.\] 
Hence $C_{u,v}^w(X)=2^{-1}\overline{C_{u,v}^w(Y)}$,
in agreement with Theorem~\ref{thm:main}.\qed
\end{example}

Since the correspondence of Theorem~\ref{thm:main} respects Graham-positivity,

\begin{corollary}\gap
\[[\beta_1^{i_1}\cdots \beta_n^{i_n}]C_{u,v}^w(X)=0 \iff  [\gamma_1^{i_1}\cdots \gamma_n^{i_n}]C_{u,v}^w(Y)=0.\] 
In particular,
$C_{u,v}^{w}(X)= 0\iff C_{u,v}^w(Y)= 0$.
\end{corollary}

Let $X'={\sf OG}(n,2n+1)$ be the \emph{maximal orthogonal Grassmannian} of $n$-dimensional subspaces of ${\mathbb C}^{2n+1}$
 that are isotropic with respect to a nondegenerate symmetric form. Also, let $Y'={\sf LG}(n,2n)$ be the \emph{Lagrangian Grassmannian}
 of $n$-dimensional subspaces of ${\mathbb C}^{2n}$
 that are isotropic with respect to a nondegenerate skew symmetric form. 
A \emph{strict partition} is an integer partition $\lambda=(\lambda_1>\lambda_2>\ldots > \lambda_{\ell})$. 
 The Schubert varieties and their (equivariant) cohomology
 classes are indexed by such $\lambda$ with $\lambda_1\leq n$ and $\ell\leq n$.
Let $\ell(\lambda)$ be the number of (nonzero) parts of a strict partition $\lambda$. 

\begin{corollary}[\emph{cf.}~Conjecture~5.1 of \cite{Monical}]
\label{cor:Harshit}
$C_{\lambda,\mu}^{\nu}(X')=2^{\ell(\nu)-\ell(\lambda)-\ell(\mu)}\overline{C_{\lambda,\mu}^{\nu}(Y')}$. 
\end{corollary}
\begin{proof}
The map $X\twoheadrightarrow X'$ that forgets all subspaces of a complete flag in $X$ except the $n$-th induces 
$H_{\sf T}(X')\hookrightarrow H_{\sf T}(X)$ sending Schubert classes to Schubert classes. The image of $\xi_\lambda(X')$
is $\xi_{w_{\lambda}}(X)$ where 
$w_{\lambda}\in {\mathcal B}_n$ is the unique ascending signed permutation beginning as $-\lambda_1,-\lambda_2,\ldots,-\lambda_{\ell}$, followed by
positive integers in increasing order.  Therefore, $C_{\lambda,\mu}^{\nu}(X')=C_{w_{\lambda},w_{\mu}}^{w_{\nu}}(X)$. 
Similarly, $C_{\lambda,\mu}^{\nu}(Y')=C_{w_{\lambda},w_{\mu}}^{w_{\nu}}(Y)$.
Hence, the result follows from Theorem~\ref{thm:main} since by definition of $w_{\lambda}$, $\ell(\lambda)=s(w_{\lambda})$. 
\end{proof}

\begin{example} Let $n=3$ and $\lambda=(3,2), \mu=(2,1)$, and $\nu=(3,2,1)$.\footnote{Hence $w_{\lambda}=\underline{3}\ \underline{2}\ 1=s_2 s_1 s_3 s_2 s_1$, $w_{\mu}=\underline{2}\ \underline{1} \ 3=s_1 s_2 s_1$ and $w_{\nu}=\underline{3} \ \underline{2} \ \underline{1}=s_1 s_2 s_1 s_3 s_2 s_1$.}
Then 
$\ell(\nu)-\ell(\lambda)-\ell(\mu)=-1$. 
Now, 
$C_{\lambda,\mu}^{\nu}(Y')=3\gamma_1^2 + 10\gamma_1\gamma_2 + 8\gamma_2^2 + 5\gamma_1\gamma_3 + 8\gamma_2\gamma_3 + 2\gamma_3^2$, 
so
\[\overline{C_{\lambda,\mu}^{\nu}(Y)}=12\beta_1^2 + 20\beta_1\beta_2 + 8\beta_2^2 + 10\beta_1\beta_3 + 8\beta_2\beta_3 + 2\beta_3^2.\]
We also have
\[{C_{\lambda,\mu}^{\nu}(X')}=6\beta_1^2 + 10\beta_1\beta_2 + 4\beta_2^2 + 5\beta_1\beta_3 + 4\beta_2\beta_3 + \beta_3^2,\] so $C_{\lambda,\mu}^{\nu}(X')=2^{-1}\overline{C_{\lambda,\mu}^{\nu}(Y')}$, agreeing with Corollary~\ref{cor:Harshit}.\qed
\end{example}

Corollary~\ref{cor:Harshit} says that the open problems of giving (Graham positive) combinatorial rules to compute $C_{\lambda,\mu}^{\nu}(X')$ and  
$C_{\lambda,\mu}^{\nu}(Y')$ are  equivalent.

Moreover, Corollary~\ref{cor:Harshit} makes exact a conjecture stated in the thesis of C.~Monical \cite[Conjecture~5.1]{Monical}. In that thesis, one also finds 
\cite[Conjecture~5.3]{Monical}, a 
conjectural recursive list of inequalities characterizing nonzeroness of $C_{\lambda,\mu}^{\nu}(X')$ (and implicitly, $C_{\lambda,\mu}^{\nu}(Y')$). 
That conjecture generalizes work of K.~Purbhoo-F.~Sottile \cite{Purbhoo.Sottile}. It is an analogue of D.~Anderson-E.~Richmond-A.~Yong
\cite{ARY} that extends work of A.~Klyachko \cite{Klyachko} and A.~Knutson-T.~Tao \cite{Knutson.Tao:99} on the eigenvalue problem
for sums of Hermitian matrices. Corollary~\ref{cor:Harshit} proves:
\begin{corollary}\emph{(\emph{cf.}~\cite[Conjecture~5.3]{Monical})}
C.~Monical's inequalities characterize $C_{\lambda,\mu}^{\nu}(X')\neq 0$ if and only if
they characterize $C_{\lambda,\mu}^{\nu}(Y')\neq 0$.
\end{corollary}

Here is another consequence of Theorem~\ref{thm:main}.
C.~Li-V.~Ravikumar \cite{Ravikumar} prove equivariant Pieri rules for  (submaximal) 
isotropic Grassmannians of classical type $B,C,D$.  Their type $B$ and $C$ rules are proved by separate geometric analyses. 
Theorem~\ref{thm:main} immediately implies a Pieri rule for type $C$ from the type $B$ rule (or vice versa).

The ``power of two'' relationship between $X'$ and $Y'$ does not hold (in any obvious way) in the Grothendieck ($K$-theory) ring of algebraic vector bundles;
see work of A.~Buch-V.~Ravikumar \cite[Examples~4.9, 5.8]{Buch.Ravikumar}. On the other hand, Theorem~\ref{thm:main} may be compared to the 
quantum cohomology result of A.~Kresch-H.~Tamvakis
\cite[Theorem~6]{Kresch.Tamvakis}.

\section{More examples of A$\cdot$B$\cdot$C$\cdot$D\MakeLowercase{s}}\label{sec:2}

\subsection{Inclusion of Dynkin diagrams}
Suppose we have an inclusion\footnote{a (multi)-graph theoretic injection that respects arrows} of (finite) Dynkin diagrams $D\hookrightarrow E$ where the nodes $1,2,\ldots,r(D)$ of $D$ are sent to the nodes $1^\circ,2^\circ,\ldots,r(D)^\circ$
of $E$, respectively. Let
\[\Delta(D)=\{\alpha_1,\ldots,\alpha_{r(D)}\}  \text{ \ and  \  $\Delta(E)=\{\beta_{1^\circ},\ldots,\beta_{r(D)^\circ},\beta_{(r(D)+1)^{\circ}},\ldots\beta_{r(E)^\circ}\}$}.\]
Given $w\in {\mathcal W}(D)$ we can unambiguously define
$w^\circ\in {\mathcal W}(E)$ by taking a reduced word $I$ for $w$ and replacing $s_{\alpha_i}$ with $s_{\beta_{i^\circ}}$ to obtain a reduced word $I^\circ$ for $w^\circ$. Let 
\[\psi_{D,E}:{\mathbb Z}[\alpha_1,\ldots,\alpha_{r(D)}]\to {\mathbb Z}[\beta_{1^\circ},\ldots,\beta_{r(D)^\circ}]\] 
be defined by $\alpha_i\mapsto\beta_{i^\circ}$. 

\begin{theorem}
\label{thm:inclusion}
$\psi_{D,E}(C_{u,v}^w(D))=C_{u^\circ,v^\circ}^{w^\circ}(E)$.
\end{theorem}
\begin{proof}
We start with the restriction version of the statement, \emph{i.e.},

\begin{claim}
\label{claim:restrictionversion}
$\psi_{D,E}(\xi_w(D)|_v)=\xi_{w^\circ}(E)|_{v^\circ}$.
\end{claim}
\noindent\emph{Proof of Claim~\ref{claim:restrictionversion}:} This 
is immediate from
(\ref{eqn:Billey}) using $I$ and $I^\circ$ respectively in computing $\xi_w(D)|_v$ and $\xi_{w^\circ}(E)|_{v^\circ}$. This is since the inclusion of Dynkin diagrams induces a canonical isomorphism of the root system of $D$ with a subroot system of $E$ that maps $\alpha_i$ to
$\beta_{i^{\circ}}$, and a canonical isomorphism of ${\mathcal W}(D)$ with the parabolic subgroup ${\mathcal W}(E)_{D}$ of ${\mathcal W}(E)$ generated by
$s_{\beta_i^\circ}$ for $1\leq i\leq r(D)$; see, \emph{e.g.}, \cite[Section~5.5]{Humphreys:reflection}.\qed

\begin{claim}
\label{claim:abc}
If $w\in {\mathcal W}(E)-{\mathcal W}(E)_D$ and $v\in {\mathcal W}(D)$ then $\xi_w(E)|_{v^{\circ}}=0$.
\end{claim}
\noindent
\emph{Proof of Claim~\ref{claim:abc}:} Since $w\in {\mathcal W}(E)-{\mathcal W}(E)_D$, by definition any reduced word for $w$ involves a $s_{\beta_{t^\circ}}$ for some $t>r(D)$. 
Fix any reduced word $I^{\circ}$ of $v^{\circ}$. Since $v^{\circ}\in {\mathcal W}(E)_D$, $I^{\circ}$ does not involve $s_{\beta_{t^{\circ}}}$. Hence no subword of $I^{\circ}$
can be a reduced word for $w$. Now the claim follows from (\ref{eqn:Billey}).\qed

Combining Claims~\ref{claim:restrictionversion} and~\ref{claim:abc} implies that for any $u,v, x\in {\mathcal W}(D)$,
\begin{equation}
\label{eqn:hhh}
\xi(E)_{u^\circ}|_{x^{\circ}} \cdot \xi(E)_{v^{\circ}}|_{x^{\circ}}=\sum_{w^{\circ}\in {\mathcal W}(E)_D} C_{u^{\circ},v^{\circ}}^{w^{\circ}}(E)\ \xi(E)_{w^{\circ}}|_{x^{\circ}} \ \ \ \forall x\in {\mathcal W}(D).
\end{equation}
By Claim~\ref{claim:restrictionversion}, for all $y\in {\mathcal W}(D)$,
\[\xi(E)_{y^{\circ}}|_{x^{\circ}} \in {\mathbb Z}_{\geq 0}[\beta_{1^{\circ}},\ldots,\beta_{r(D)^{\circ}}].\]
Therefore by this nonnegativity and W.~Graham's theorem (\ref{eqn:Graham}), it must be that 
\[C_{u^{\circ},v^{\circ}}^{w^{\circ}}(E)\in {\mathbb Z}_{\geq 0}[\beta_{1^{\circ}},\ldots,\beta_{r(D)^{\circ}}].\] 
Therefore, it makes sense to apply $\psi^{-1}_{D,E}$ to both sides of (\ref{eqn:hhh}) to obtain
\begin{equation}
\label{eqn:jjk}
\xi(D)_u|_x \cdot \xi(D)_v|_x=\sum_{w\in {\mathcal W}(D)} \psi_{D,E}^{-1}(C_{u^{\circ},v^{\circ}}^{w^{\circ}}(E))\ \xi(D)_w|_{x} \ \ \ \forall x\in {\mathcal W}(D).
\end{equation}
We can now conclude as in the proof of Theorem~\ref{thm:main}.
By uniqueness of the structure constants, (\ref{eqn:jjk}) asserts 
\[\psi_{D,E}^{-1}(C_{u^{\circ},v^{\circ}}^{w^{\circ}}(E))=C_{u,v}^w(D).\] 
Apply $\psi_{D,E}$ to both sides to conclude the proof.
\end{proof}

\begin{example}
The Dynkin diagram for $F_4$ is \ \ 
$\begin{picture}(36,10)
\thicklines
\put(0,-4){$1$}
\put(0,5){$\circ$}
\put(5,8){\line(1,0){15}}
\put(19,-4){$2$}
\put(19,5){$\circ$}
\put(24,7){\line(1,0){15}}
\put(24,9){\line(1,0){15}}
\put(27,5){$<$}
\put(38,-4){$3$}
\put(38,5){$\circ$}
\put(43,8){\line(1,0){15}}
\put(57,-4){$4$}
\put(57,5){$\circ$ }
\end{picture}$ \ \ \ \ \ \  \ \  \ .  Now, there is an embedding of $D=B_3$ into $E=F_4$ given by $1\mapsto 1^{\circ}=2, 2\mapsto 2^{\circ}=3, 3\mapsto 3^{\circ}=4$. One computes that
\[C_{s_1 s_2 s_1, s_2 s_3 s_1}^{s_1 s_2 s_3 s_1}(B_3)=2\beta_1^2 +3\beta_1 \beta_2 + \beta_2^2 \text{ \ and \ }
C_{s_2 s_3 s_2, s_3 s_4 s_2}^{s_{2} s_3 s_4 s_2}(F_4)=2\zeta_2^2 +3\zeta_2 \zeta_3 + \zeta_3^2.\]
These are equal after $\beta_i\mapsto \zeta_{i+1}$ for $1\leq i\leq 3$, in agreement with the Theorem~\ref{thm:inclusion}. 
\qed
\end{example}

Besides being computationally useful, Theorem~\ref{thm:inclusion} is a guiding property in the search for an eventual combinatorial rule
for $C_{u,v}^w$. See \cite[Section~5.2]{Thomas.Yong:comin} for hints of this in the root-system uniform (non-equivariant) rule for the special case of minuscule
flag varieties.

There are coincidences between types $B_n$ and $D_{n+1}$, since the Dynkin diagram of the former is the
``folding'' of the Dynkin diagram for the latter:
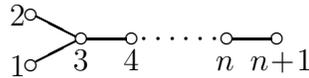
\[\!\!\!\!\!\!\!\!\!\!\!\!\!\!\!\!\!\!\!\!\!\!\!\!\!\!\!\begin{picture}(30,18)
\thicklines
\put(-5,-8){$1$}
\put(-5, 11){$2$}
\put(0,12){$\circ$}
\put(0,-7){$\circ$}
\put(20,5){\line(-2,-1){15}}
\put(20,7){\line(-2,1){15}}
\put(19,-6){$3$}
\put(19,3){$\circ$}
\put(24,6){\line(1,0){15}}
\put(38,-6){$4$}
\put(38,3){$\circ$}
\put(44,3){$\cdots$}
\put(60,3){$\cdots$}
\put(73,-6){$n$}
\put(74,3){$\circ$}
\put(79,6){\line(1,0){15}}
\put(93,3){$\circ$}
\put(86,-6){$n\!+\!1$}
\end{picture}\] 

\smallskip

\begin{example}
$C_{s_1 s_2 s_1,s_1 s_2 s_1}^{s_1 s_2 s_1}(B_2)=\beta_1(2\beta_1+\beta_2)(\beta_1+\beta_2)$. It is natural to compare
$s_1 s_2 s_1 \in {\mathcal W}(B_2)$ with $s_1 s_3 s_2 \in {\mathcal W}(D_3)$. Indeed, 
\[C_{s_1 s_3 s_2,s_1 s_3 s_2}^{s_1 s_3 s_2}(D_3)=\delta_1(\delta_1+\delta_2+\delta_3)(\delta_1+\delta_3)\] 
equals $C_{s_1 s_2 s_1,s_1 s_2 s_1}^{s_1 s_2 s_1}(B_2)$
under the ``folding substitution'' $\delta_1,\delta_2\mapsto \beta_1$ and $\delta_3\mapsto \beta_2$. \qed
\end{example}

Such a 
substitution gives a correspondence between ${\sf OG}(n,2n+1)$ restrictions and a subset of restrictions of ${\sf OG}(n+1,2n+2)$ (the 
maximal isotropic Grassmannian of type $D_{n+1}$); see \cite[Remark 5.7]{Graham.Kreiman} and the references therein. By the A$\cdot$B$\cdot$C$\cdot$D 
argument as in Theorem~\ref{thm:main}, one obtains a correspondence of structure coefficients. Unfortunately, this correspondence is not
true in general, even for restrictions:

\begin{example} One calculates that
\[\xi_{s_2 s_1 s_2 s_3}|_{s_2 s_1 s_2 s_3}(B_3)=4\beta_1^3 \beta_2+10\beta_1^2\beta_2^2+2\beta_1^2\beta_2 \beta_3 +8 \beta_1\beta_2^3+ 3\beta_1\beta_2^2\beta_3+2\beta_2^4+\beta_2^3\beta_3.\]
By direct search, there is no $\xi_v|_v(D_4)$ which, after the folding substitution $\delta_1,\delta_2\mapsto \beta_1, \delta_3\mapsto \beta_2, \delta_4\mapsto\beta_3$, has even the same monomial support as 
$\xi_{s_2 s_1 s_2 s_3}|_{s_2 s_1 s_2 s_3}(B_3)$.\qed
\end{example}

\subsection{Nonvanishing}

The result is known, \emph{cf}.~\cite[Corollary 4.5]{Billey} which credits \cite{KK1}. We include a proof to be self-contained.
\begin{proposition} \label{prop:interval} 
$\xi_w|_{v}\neq 0$ for all $w\leq v\leq w_0$.
\end{proposition}
\begin{proof}
Suppose $v\leq v'$ and fix a reduced word $I'$ for $v'$. By the subword property of Bruhat order, there is a subword $I$ of $I'$ which is reduced for $v$. Any subword
$J$ of $I$ that is a reduced word for $w$ is also a subword of $I'$. Thus, by (\ref{eqn:Billey}), 
any monomial appearing in $\xi_w|_v$ associated to $J$ corresponds to a maybe different monomial (in the positive roots) in $\xi_w|_{v'}$. 
Now use that (\ref{eqn:Billey}) says $\xi_w|_{w}$ is a nonzero monomial. 
\end{proof}

\begin{conjecture}[A$\cdot$B$\cdot$C$\cdot$D version of Proposition~\ref{prop:interval}]
\label{conj:anotherinterval}
Assume $C_{u,v}^{w}\neq 0$.
\begin{itemize}
\item[(I)] $C_{u,s_{\alpha}v}^{w}\neq 0$ when $v<s_{\alpha}v\leq w$ and $\alpha\in \Delta$.
\item[(II)] If $\ell(w)<\ell(u)+\ell(v)$ then there exists $s_{\alpha}$ ($\alpha\in \Delta$) with $s_{\alpha}v<v$ such that $C_{u,s_{\alpha}v}^w\neq 0$.
\end{itemize}
\end{conjecture}

\begin{example}
In Conjecture~\ref{conj:anotherinterval}, the existential quantification in (II) is needed. In type $B_3$,
\[C_{s_2 s_3, s_1s_3}^{s_2 s_1 s_3}=\beta_2+\beta_3, \text{   \ but \ $C_{s_2 s_3, s_1(s_1 s_3)}^{s_2 s_1 s_3}=0$.}\] 
Now, 
$C_{s_2 s_3, s_3(s_1 s_3)}^{s_2 s_1 s_3}=1$, as predicted. \qed
\end{example}
We exhaustively checked Conjecture~\ref{conj:anotherinterval} for $A_4, B_3$ and $G_2$ and for many examples in $A_5, B_4$ and $F_4$.
Conjecture~\ref{conj:anotherinterval} holds for Grassmannians, where it plays a key role in \cite{ARY}, which connects \cite{Friedland} to the equivariant structure coefficients. C.~Monical's extension, discussed in Section~\ref{sec:1}, motivates this conjecture.

\begin{example}
There is no ``righthand version'' of either part of Conjecture~\ref{conj:anotherinterval}. For (I),
\[C_{s_1s_2,s_1}^{s_1s_2s_1}(A_2)=1 \text{ \ but $C_{s_1s_2, (s_1)s_2}^{s_1s_2s_1}(A_2)=0$.}\] 
Whereas for (II), $C_{s_1s_2,s_2s_1}^{s_1s_2s_1}(A_2)=\alpha_1+\alpha_2$
yet $C_{s_1s_2,(s_2s_1)s_1}^{s_1s_2s_1}(A_2)=0$.\qed
\end{example}

Proposition~\ref{prop:interval} implies that, for the classical types, the decision problem ${\tt Restriction}$ ``$\xi_{w}|_v\neq 0?$'' is in the class
${\sf P}$ of polynomial time problems.\footnote{For complexity purposes, the expectional types are ignored since they are finite in number.} This is since there is a polynomial time \emph{tableau criterion} for deciding if $w\leq v$ for corresponding Weyl groups; see \cite[Chapters 2, 8]{Bjorner.Brenti} (here the input size is bounded by a polynomial in $r$). The A$\cdot$B$\cdot$C$\cdot$D version of this claim 
concerns the decision problem ${\tt Nonvanishing}$:  ``$C_{u,v}^w\neq 0$?'' given input $u,v,w\in {\mathcal W}$ (in one line notation). 
\begin{conjecture}
\label{conj:ABCDnonvanishing}
For each classical Lie type, {\tt Nonvanishing}$\ \in{\sf P}$.
\end{conjecture}

Conjecture~\ref{conj:ABCDnonvanishing} is highly speculative. That said, it holds for Grassmannians \cite{Adve.Robichaux.Yong:P}. In our opinion, this conjecture is related to the (testable) Conjecture~\ref{conj:second} given below.

\subsection{Counterexamples to A$\cdot$B$\cdot$C$\cdot$D} It is interesting to study situations where A$\cdot$B$\cdot$C$\cdot$D is (seemingly) false. For instance, here is a true statement about restrictions:  

\begin{theorem}[Monotonicity]
\label{prop:monotonicity}
If $w\leq v\leq v'$ then 
$\xi_w|_{v'} - \xi_w|_v \in {\mathbb Z}_{\geq 0}[\alpha_1,\ldots,\alpha_r]$.
\end{theorem}
\begin{proof}
It suffices to prove this when $v'$ covers $v$. Then $$ \xi_w \cdot \xi_v = C_{w,v}^{v}\ \xi_v+ C_{w,v}^{v'}\ \xi_{v'}+ \sum\limits_{\tilde{v}\geq v,\tilde{v}\neq v'} C_{w,v}^{\tilde{v}} \ \xi_{\tilde{v}}. $$
Restricting the above equation at $v'$, using (\ref{eqn:nonzeroness}) and the fact (\ref{eqn:Knutson}) that $C_{w,v}^{v}=\xi_w|_v$, we get
\begin{equation}
\xi_w|_{v'} \cdot \xi_v|_{v'} = \xi_w|_v \cdot \xi_v|_{v'}+ C_{w,v}^{v'}\xi_{v'}|_{v'}. 
\end{equation}
By Proposition~\ref{prop:interval}, $\xi_v|_{v'}\neq 0$, therefore,
\begin{equation}
\xi_w|_{v'} - \xi_w|_v = \frac{\xi_{v'}|_{v'}}{\xi_v|_{v'}}C_{w,v}^{v'}.
\end{equation}

Fix a reduced word $s_{i_1}s_{i_2}\cdots s_{i_m}$ for $v'$. By the \emph{strong exchange property} of Bruhat order \cite[Section~5.8]{Humphreys:reflection}, there exists a
\emph{unique} $1\leq k\leq m$ such that $s_{i_1}\cdots s_{i_{k-1}}s_{i_{k+1}}\cdots s_{i_m}$ ($s_{i_k}$ omitted) is a reduced word of $v$. Therefore by (\ref{eqn:Billey}), 
\[\frac{\xi_{v'}|_{v'}}{\xi_v|_{v'}} = s_{i_1} \ldots s_{i_{k-1}} \cdot \alpha_{i_k}\in \Phi^+.\] 
Hence, by (\ref{eqn:Graham}),
$\xi_w|_{v'} - \xi_w|_v = (s_{i_1} \ldots s_{i_{k-1}}\cdot \alpha_{i_k})C_{w,v}^{v'}\in {\mathbb Z}_{\geq 0}[\alpha_1,\ldots,\alpha_r]$, as desired.
\end{proof}

\begin{example}[Monotonicity counterexample] Thus, it is tempting to conjecture that if $u,v,w\in {\mathcal W}$ and $s_{\alpha}$ is a simple reflection such that 
$u\leq us_{\alpha}:=u'$ and $w\leq  ws_{\alpha}:=w'$ then 
$C_{u',v}^{w'}-C_{u,v}^w\in {\mathbb Z}_{\geq 0}[\alpha_1,\ldots,\alpha_r]$. In particular, this would imply
$c_{u',v}^{w'}\geq c_{u,v}^w$. However, that is false in general. For instance in $A_5$ if $u=351624, v=214356, w=631524$ and $s=s_3$ 
$c_{u,v}^w=c_{351624,214356}^{631524}=1$ but 
$c_{u',v}^{w'}=c_{356124,214356}^{635124}=0$.
\qed
\end{example}

A theorem of A.~Arabia \cite{Arabia} states:
\[\alpha \text{\ divides \ } \xi_{w}|_{s_\alpha v}-\xi_w|_{v}.\]
(In general, this is the condition of \cite{GKM} that describes the image of (\ref{eqn:injective}).)

\begin{example}[Divisibility counterexample]
Does $\alpha \text{\ divide \ } C_{s_{\alpha}u,v}^{s_{\alpha}w}-C_{u,v}^w$?  This is false in general. In type $A_3$ let
$u=s_3$, $v=s_2 s_3 s_1$ and  $w=s_2 s_3 s_1$. Then $C_{u,v}^{w} = \alpha_2 +\alpha_3$. Let $s_{\alpha}=s_1$ and hence $s_{\alpha}u=s_1u=s_1 s_3 , s_{\alpha}w=s_1w=s_1 s_2 s_3 s_1$. Now $C_{s_{\alpha}u,v}^{s_{\alpha}w} = \alpha_1 + \alpha_2$, and thus
$C_{s_{\alpha}u,v}^{s_{\alpha}w} - C_{u,v}^{w} = \alpha_1 - \alpha_3$ 
is neither $\alpha$-positive nor divisible by $\alpha_1$.\qed
\end{example}

A number of other  simple variations on monotonicity and divisibility  are false as well.
Can the A$\cdot$B$\cdot$C$\cdot$Ds for monotonicity/divisibility be realized, under a hypothesis?

\subsection{Newton polytopes}
The \emph{Newton polytope} of 
\[f=\sum_{(n_1,\ldots,n_r)\in {\mathbb Z}_{\geq 0}^r} c_{n_1,\ldots,n_r}\prod_{j=1}^r \alpha_j^{n_j}\in {\mathbb R}[\alpha_1,\ldots,\alpha_r]\] 
is ${\sf Newton}(f):={\sf conv}\{(n_1,\ldots,n_r): c_{n_1,\ldots,n_r}\neq 0\}\subseteq {\mathbb R}^r$.

\begin{proposition}
Let $w\in {\mathcal W}$ and $w\leq v\leq v'$. Then ${\sf Newton}(\xi_w|_v)\subseteq {\sf Newton}(\xi_w|_{v'})$.
\end{proposition}
\begin{proof}
This is immediate from Theorem~\ref{prop:monotonicity}.
\end{proof}

$f$ has \emph{saturated Newton polytope} (SNP) \cite{Monical.Tokcan.Yong} if
$c_{n_1,\ldots,n_r}\neq 0 \iff (n_1,\ldots,n_r)\in {\sf Newton}(f)$. 
\begin{conjecture}
\label{conj:first}
Let $v,w\in {\mathcal W}$, then $\xi_w|_v$ has SNP.
\end{conjecture}

\begin{conjecture}[A$\cdot$B$\cdot$C$\cdot$D version of Conjecture~\ref{conj:first}]
\label{conj:second}
Let $u,v,w\in {\mathcal W}$, then $C_{u,v}^w$ has SNP.
\end{conjecture}
We exhaustively checked these conjectures for $A_4, B_3, D_4, G_2$ and many examples in $A_6$ and $B_4$. 
A proof of either conjecture for Grassmannians would be interesting.

SNP is  connected to computational complexity in \cite[Section~1]{Adve.Robichaux.Yong}.
We suspect the concrete SNP claim of Conjecture~\ref{conj:second} is the combinatorial harbinger of the {\sf P} assertion of 
 Conjecture~\ref{conj:ABCDnonvanishing}.  Let {\tt Schubert} be the decision problem ``$(n_1,\ldots,n_r)\in {\sf Newton}(C_{u,v}^w)$?'', given input $u,v,w\in {\mathcal W}$ and $(n_1,\ldots,n_r)\in {\mathbb Z}_{\geq 0}^r$.
It is reasonable to conjecture existence of:
\begin{itemize}
\item a combinatorial rule for $C_{u,v}^w$ that moreover implies counting $C_{u,v}^w$ is a problem in the counting complexity class $\#{\sf P}$, and
\item  a halfspace description of ${\sf Newton}(C_{u,v}^w)$ where each \emph{individual} inequality can be checked in polynomial time (even if there are exponentially many
inequalities). 
\end{itemize}
Conjecture~\ref{conj:second} would then imply ${\tt Schubert}\in{\sf NP}\cap {\sf coNP}$. \emph{Often} problems in ${\sf NP}\cap {\sf coNP}$ are in fact in ${\sf P}$ (see \cite[Section~1.2]{Adve.Robichaux.Yong} for a discussion). 
${\tt Schubert}\in {\sf P}$ implies the important case of Conjecture~\ref{conj:ABCDnonvanishing} for the non-equivariant $c_{u,v}^w$ is true.

\section{Comparisons to Schubert polynomial theory}\label{sec:3}
The theory of \emph{Schubert polynomials}, introduced by A.~Lascoux and M.-P.~Sch\"utzenberger \cite{LS}, is influential in the conversation of positivity in Schubert calculus.

 These polynomials ``lift'' the Schur polynomials from the ring of symmetric polynomials to the ring of all polynomials. The study of Schur polynomials is backed by an extensive literature on Young tableaux, from which one obtains the Littlewood-Richardson rule. Thus one might hope for an analogous theory for Schubert polynomials; this remains unrealized. For the purposes of our discussion, let us call this the ``lifting dream''.

Theorem~\ref{thm:main} generalizes the identity 
\[c_{u,v}^w(X)=2^{s(w)-s(u)-s(v)}c_{u,v}^w(Y).\] 
This seems to have been first stated in \cite[(3.2)]{Bergeron.Sottile}, who rely on the
Schubert polynomials for classical groups of S.~Billey-M.~Haiman \cite{Billey.Haiman}. Similarly, Corollary~\ref{cor:Harshit} generalizes the equality 
\begin{equation}
\label{eqn:Pragacz}
c_{\lambda,\mu}^{\nu}(X')=2^{\ell(\nu)-\ell(\lambda)-\ell(\mu)}c_{\lambda,\mu}^{\nu}(Y'),
\end{equation}
which is a consequence of
P.~Pragacz \cite[Theorem~6.17]{Pragacz} on the Schubert calculus interpretation of the Schur $Q-$, $P-$ functions. Theorem~\ref{thm:main} also follows from  T.~Ikeda-L.~Mihalcea-H.~Naruse \cite{IMN} who give an equivariant  generalization of the polynomials of \cite{Billey.Haiman}.

Over the past three decades, within algebraic combinatorics, the emphasis has been on the Schubert polynomial rather than the list of many restrictions.\footnote{Not that the two viewpoints are unrelated: in type $A$ for example, one can compute the restrictions as certain specializations of the double Schubert polynomials; see, e.g., \cite{Billey}.} Our proof replaces the effort of the polynomial constructions \cite{Pragacz, Billey.Haiman, IMN} with the general geometric result (\ref{eqn:injective}). This work suggests   
A$\cdot$B$\cdot$C$\cdot$D as an alternative to the ``lifting dream'' and one that opens up some new and testable possibilities. 

Is there concrete evidence for preferring one approach to the other? For example, can one give an
A$\cdot$B$\cdot$C$\cdot$D proof of S.~Robinson's equivariant Pieri rule for ${\sf GL}_n/{\sf B}$ \cite{Robinson}? Can one give a Schubert polynomial (in this case, factorial Schur polynomial)  proof of one or more of the combinatorial rules \cite{Knutson.Tao, Kreiman, Thomas.Yong:Eq} by giving an equivariant version of Schensted insertion? Based on earlier conversation of the third author with H.~Thomas, this latter question seems quite nontrivial.
   
\section*{Acknowledgements}
This paper was stimulated by the thesis of Cara Monical \cite{Monical}; we made use of her related code.
We thank the organizers of the Ohio State Schubert calculus conference (May 2018), which facilitated consultation with experts about the conjectures of \cite{Monical}.
Dylan Rupel provided helpful comments about an earlier draft. Sue Tolman explained to us the seminal role of \cite{Atiyah.Bott} in equivariant symplectic geometry.
AY was partially supported by an NSF grant, a U$\cdot$I$\cdot$U$\cdot$C Campus Research Board grant, and a Simons Collaboration Grant. 
This material is based upon work of CR supported by the National Science Foundation Graduate Research Fellowship Program under Grant No. DGE -- 1746047.


\begin{thebibliography}{99}


\bibitem{Adve.Robichaux.Yong} A.~Adve, C.~Robichaux and A.~Yong, \emph{Complexity, combinatorial positivity, and Newton polytopes}, preprint, 2018. \textsf{arXiv:1810.10361}

\bibitem{Adve.Robichaux.Yong:P} \bysame, \emph{Vanishing of Littlewood-Richardson polynomials is in ${\sf P}$}, Comput. Complexity 28 (2019), no. 2, 241--257.

\bibitem{ARY} D.~Anderson, E.~Richmond and A.~Yong, \emph{Eigenvalues of Hermitian matrices and equivariant cohomology of Grassmannians}, Compos. Math. 149 (2013), no. 9, 1569--1582.



\bibitem{Soergel} H. Andersen, J. Jantzen, and W. Soergel, \emph{Representations of quantum groups at a pth root of unity and of semisimple groups in characteristic $p$: independence of $p$}, Astrisque No. 220 (1994), 321 pp.

\bibitem{Arabia}
A. Arabia, \emph{Cohomologie $T$-\'equivariant de la vari\'et\'e de drapeaux d'un groupe de Kac-Moody}, Bull. Math. Soc. France 117 (1989), 129--165.

\bibitem{Atiyah.Bott} 
M.~F.~Atiyah and R.~ Bott, \emph{The Yang-Mills equations over Riemann surfaces}, Philos. Trans. Roy. Soc. London Ser. A 308 (1982), no. 1505, 523--615.

\bibitem{Bergeron.Sottile} 
N.~Bergeron and F.~Sottile, \emph{A Pieri-type formula for isotropic flag manifolds}, Trans. Amer. Math. Soc. 354 (2002), no. 7, 2659--2705. 

\bibitem{Billey} S. Billey, \emph{Kostant polynomials and the cohomology ring for $G/B$}, Duke Math. J. 96 (1999), no. 1, 205--224.

\bibitem{Billey.Haiman} S.~Billey and M.~Haiman, \emph{Schubert polynomials for the classical groups}, J. Amer. Math. Soc. 8 (1995), no. 2, 443--482.

\bibitem{Bjorner.Brenti} A.~Bj\"orner and F.~Brenti, \emph{Combinatorics of Coxeter groups}, Graduate Texts in Mathematics, 231. Springer, New York, 2005.

\bibitem{Buch.Ravikumar} A.~Buch and V.~Ravikumar, \emph{Pieri rules for the $K$-theory of cominuscule Grassmannians}, J. Reine Angew. Math. 668 (2012), 109--132.

\bibitem{Friedland} S.~Friedland, \emph{Finite and infinite dimensional generalizations of Klyachko's theorem}, Linear Algebra Appl., {\bf 319} (2000), 
3--22.

\bibitem{Fulton} W.~Fulton, \emph{Young tableaux. With applications to representation theory and geometry}, London Mathematical Society Student Texts, 35. Cambridge University Press, Cambridge, 1997. 

\bibitem{GKM} M.~Goresky, R.~Kottwitz and R.~MacPherson, \emph{Equivariant cohomology, Koszul duality, and the localization theorem}, Invent. Math. 131 (1998), no. 1, 25--83.

\bibitem{Graham} W.~Graham, \emph{Positivity in equivariant Schubert calculus}, Duke Math. J. 109 (2001), no. 3, 599--614.

\bibitem{Graham.Kreiman} W.~Graham and V.~Kreiman, \emph{Excited Young diagrams, equivariant $K$-theory, and Schubert varieties}, 
Trans. Amer. Math. Soc. 367 (2015), no. 9, 6597--6645.

\bibitem{Humphreys:reflection} J.~Humphreys, \emph{Reflection groups and Coxeter groups}. 
Cambridge Studies in Advanced Mathematics, 29. Cambridge University Press, Cambridge, 1990. 

\bibitem{IMN} T.~Ikeda, L.~C.~Mihalcea and H.~Naruse, \emph{Double Schubert polynomials for the classical groups}, 
Adv. Math. 226 (2011), no. 1, 840--886.

\bibitem{Klyachko} A.~A.~Klyachko, \emph{Stable vector bundles and Hermitian operators},
Selecta Math. (N.S.) {\bf 4} (1998), 419--445.

\bibitem{Knutson:descent} A.~Knutson, \emph{Descent-cycling in Schubert calculus}, Experiment. Math. 10 (2001), no. 3, 345--353.

\bibitem{Knutson} \bysame, \emph{A Schubert calculus recurrence from the noncomplex W-action on G/B}, preprint, 2003.
{\sf arXiv:math.0306304}

\bibitem{Knutson:patches} \bysame, \emph{Schubert patches degenerate to subword complexes}, 
Transform. Groups 13 (2008), no. 3-4, 715--726.

\bibitem{Knutson.Tao:99} A.~Knutson and T.~Tao, \emph{The honeycomb model of $GL_n({\mathbb C})$ tensor products I: proof of the saturation 
conjecture}, J.~Amer.~Math.~Soc. {\bf 12} (1999), 1055--1090.


\bibitem{Knutson.Tao} \bysame, \emph{Puzzles and (equivariant) cohomology of Grassmannians}, Duke Math. J. 119 (2003), no. 2, 221--260.

\bibitem{KK1} B. Kostant and S. Kumar, \emph{The nil Hecke ring and cohomology of $G/P$ for a Kac-Moody group $G$}, 
Adv.~Math. {\bf 62}(1986), no. 3, 187--237.
\bibitem{KK2} \bysame, \emph{$T$-equivariant $K$-theory of generalized flag varieties}, J. Differential Geom. 32
(1990), no. 2, 549--603.

\bibitem{Kreiman} V.~Kreiman, \emph{Equivariant Littlewood-Richardson skew tableaux}, Trans. Amer. Math. Soc. 362 (2010), no. 5, 2589--2617.

\bibitem{Kresch.Tamvakis} A.~Kresch and H.~Tamvakis, \emph{Quantum cohomology of orthogonal Grassmannians}, 
Compos. Math. 140 (2004), no. 2, 482--500.

\bibitem{LS} A.~Lascoux and M.-P.~Sch\"{u}tzenberger, \emph{Polyn\^{o}mes de Schubert},  C. R. Acad. Sci. Paris S\'{e}r. I Math. 294 (1982), no. 13, 447--450. 

\bibitem{Ravikumar} C.~Li and V.~Ravikumar, \emph{Equivariant Pieri rules for isotropic Grassmannians},
 Math. Ann. 365 (2016), no. 1-2, 881--909.

\bibitem{Molev} A.~I.~Molev, \emph{Littlewood-Richardson polynomials}, J. Algebra 321 (2009), no. 11, 3450--3468.

\bibitem{Monical} C.~Monical, \emph{Polynomials in algebraic combinatorics}, Ph.D thesis, University of Illinois at Urbana-Champaign, 2018.

\bibitem{Monical.Tokcan.Yong} C.~Monical, N.~Tokcan and A.~Yong, \emph{Newton polytopes in algebraic combinatorics}, preprint, 2017.
\textsf{arXiv:1703.02583}

\bibitem{Pragacz} P.~Pragacz, \emph{Algebro-geometric applications of Schur $S$- and $Q$-polynomials,} Topics in invariant theory (Paris, 1989/1990), 130--191, Lecture Notes in Math., 1478, Springer, Berlin, 1991.

\bibitem{Purbhoo.Sottile} K.~Purbhoo and F.~Sottile,
\emph{The recursive nature of cominuscule Schubert calculus}, Adv.~Math., {\bf 217}(2008), 1962--2004.

\bibitem{Robinson} S.~Robinson, \emph{A Pieri-type formula for $H^\ast_T({\rm SL}_n(\Bbb C)/B)$}, J. Algebra 249 (2002), no. 1, 38--58.

\bibitem{Thomas.Yong:comin} H.~Thomas and A.~Yong, \emph{A combinatorial rule for (co)minuscule Schubert calculus}, 
Adv. Math. 222 (2009), no. 2, 596--620.

\bibitem{Thomas.Yong:Eq} \bysame, \emph{Equivariant Schubert calculus and jeu de taquin},
Ann. Inst. Fourier (Grenoble) 68 (2018), no. 1, 275--318.

\bibitem{Tymoczko} J.~Tymoczko, \emph{Billey's formula in combinatorics, geometry, and topology}, Schubert calculus Osaka 2012, 499--518, Adv. Stud. Pure Math., 71, Math. Soc. Japan, [Tokyo], 2016.

\end{thebibliography}
\end{document}